\documentclass[10 pt, leqno]{article}
\date{}
\usepackage{amssymb,amsbsy,amsmath,amsfonts,amssymb,amscd, mathrsfs, amsthm}
\usepackage[english]{babel}
\usepackage[T1]{fontenc}
\usepackage{indentfirst}

\makeatletter
\@addtoreset{equation}{section}
\makeatother

\newtheorem{statement}{}[section]
\newtheorem{theoreme}[statement]{Theorem}
\newtheorem{lemme}[statement]{Lemma}

\newtheorem{proposition}[statement]{Proposition}
\newtheorem{definition}[statement]{Definition}

\newcommand\C{\mathbb C}
\newcommand\N{\mathbb N}
\newcommand\R{\mathbb R}
\newcommand\T{\mathbb T}
\newcommand\D{\mathbb D}
\newcommand\Z{\mathbb Z}
\newcommand\e{{\rm e}}

\newcommand\eps{\varepsilon}
\newcommand\ind{{\rm 1\kern-.30em I}}

\let\phi=\varphi

\title{\bf Approximation numbers of composition operators on the Hardy space of the infinite polydisk}
\author{\it Daniel~Li,  Herv\'e~Queff\'elec, L.~Rodr{\'\i}guez-Piazza}

\date{\footnotesize \today}

\begin{document}

\maketitle

\noindent{\bf Abstract.} We study the composition operators of the Hardy space on $\D^\infty \cap \ell_1$, the $\ell_1$ part of the infinite polydisk, and the 
behavior of their approximation numbers.

%\vspace{5mm}
%%%%%%%%%%%%%%%%%%%%%%%%%%%%%%%%%%%%%%%%%%
\section{Introduction}

Recently, in \cite{BLQR}, we investigated approximation numbers $a_{n}(C_\varphi), n\geq 1$, of composition operators $C_\varphi,\ C_{\varphi}(f)=f\circ \varphi$, on the 
Hardy or Bergman spaces $H^{2}(\Omega)$, $B^{2}(\Omega)$ over a bounded symmetric domain $\Omega\subseteq \C^d$. Assuming that 
$\varphi(\Omega)$ has non-empty interior, one of the main results of this study was the following theorem.

\begin{theoreme} [\cite{BLQR}]
Let $C_\varphi \colon H^{2}(\Omega)\to H^{2}(\Omega)$ be compact. Then:  \par\smallskip

$1)$ we always have $a_{n}(C_\varphi)\geq c \, \e^{- C \, n^{1/d}}$ where $c, C$ are positive constants; \par\smallskip

$2)$ if $\Omega$ is a product of balls and if $\varphi(\Omega)\subseteq r \, \Omega$ for some $r<1$, then:
\begin{displaymath} 
a_{n} (C_\varphi) \leq C \, \e^{-c \, n^{1/d}} \, .
\end{displaymath}
\end{theoreme}

As a result, the minimal decay of approximation numbers is slower and slower as the dimension $d$ increases, which might lead one to think that, in 
infinite-dimension, no compact composition operators can exist, since their approximation numbers will not tend to $0$. After all, this is the case for the Hardy 
space of a half-plane, which supports no compact composition operator (\cite{MAT}, Theorem 3.1; in \cite{ELLIOT-JURY}, it is moreover proved that 
$\| C_\phi \|_e = \|C_\phi\|$ as soon as $C_\phi$ is bounded; see also \cite{Shap-Smith} for a necessary and sufficient condition for $H^2 (\Omega)$ has 
compact composition operators, where $\Omega$ is a domain of $\C$). We will see that this is not quite the case here, even though the decay will be severely 
limited. In particular, we will never have a decay of the form $C\, \e^{- c \, n^{\delta}}$ for some $c, C, \delta > 0$. 
%\bigskip

%%%%%%%%%%%%%%%%%%%%%%%%%%%%%%%%%%%%%%%%%%%%%%%%%%%%%%%%%%%%%%%%%%%%%%%%%
\section{Framework and reminders}

%We will limit ourselves to the following setting: 
%
%$$\Omega=\D^{\infty}\cap \ell^1$$ is the \textit{open} subset of $\ell^1$   formed by those sequences 
%$$z=(z_n)_{n\geq 1} \hbox{\ such that}\  |z_n|<1\  \forall n \hbox{\ and}\ \sum_{n=1}^\infty |z_n|<\infty.$$
%We begin with a simple  obervation: the function $\rho(z)=\sup_{j}|z_j|=\Vert z\Vert_\infty$ is continuous on $\Omega$. Indeed
%$$|\rho(z)-\rho(w)|\leq \Vert z-w\Vert_\infty \leq \Vert z-w\Vert_1.$$ This implies in particular that $\Omega=\rho^{-1}(-\infty,1)$ is  open in $\ell^1$ as 
%claimed and that, if $K\subset \Omega$ is compact, then 
%$\sup_{z\in K} \rho(z)<1.$\\
%Now, on this open set $\Omega=\D^\infty\cap\ell^1$, we can and will do differential calculus in the ordinary sense.\\ 
%In fact, the choice of the Banach space will play a relatively insignificant role, and we might as well have chosen $\Omega=\D^{\infty}\cap \ell^2$ as this is 
%customary in connection with Dirichlet series (\cite{COGA}). But the choice of $\ell^1$ appears technically simpler, in that the automorphisms of the 
%corresponding open set $\Omega$ will act transitively on $\Omega$ and generate bounded composition operators on $H^{2}(\Omega)$. 
%

%\subsection{ Hardy spaces on $\Omega$}
\subsection{ Hardy spaces on $\D^\infty$}

Let $\T = \partial \D$ be the unit circle of the set of complex numbers. We consider $\T^\infty$ and equip it with its Haar measure $m$. This is a compact 
Abelian group  with dual $\Z^{(\infty)}$, the set of eventually zero sequences $\alpha = (\alpha_j)_{j\geq 1}$ of integers. We denote 
$L^{2}_{\N^{(\infty)}} (\T^\infty)$ the Hilbert subspace of $L^{2} (\T^\infty)$ formed by the functions $f$ whose Fourier spectrum is contained in 
$\N^{(\infty)}$:
\begin{displaymath}
\qquad\quad  \widehat{f}(\alpha):=\int_{\T^\infty} f (z) \, \overline{z}^{\alpha}\, dm (z)=0 \qquad \text{if } \alpha\notin  \N^{(\infty)} \, .
\end{displaymath}
The set $E := \N^{(\infty)}$ is called the \emph{narrow cone of Helson}, and we also denote  $L^{2}_{\N^{(\infty)}} (\T^\infty) = L^{2}_{E}(\T^\infty)$. 
Any element of that subspace can be formally written as:
\begin{displaymath} 
\hskip 50 pt f = \sum_{\alpha\geq 0} c_\alpha \, e_\alpha \qquad 
\text{with } c_\alpha = \widehat{f}(\alpha) \quad \text{and} \quad \sum_{\alpha \geq 0} |c_\alpha|^2 < \infty \, .
\end{displaymath} 
Here, $(e_\alpha)_{\alpha\in \Z^{(\infty)}}$ is the canonical basis of $L^{2}(\T^\infty)$ formed by characters, and accordingly 
$(e_\alpha)_{\alpha\in \N^{(\infty)}}$ is the canonical basis of $L^{2}_E (\T^\infty)$.  
\medskip

Now we consider $\Omega_2 = \D^\infty \cap \ell_2$. 
\par\smallskip

Any $f\sim \sum_{\alpha \geq 0} c_\alpha \, e_\alpha\in L^2_E (\T^\infty)$ defines an analytic function on the infinite-dimensional Reinhardt domain 
$\Omega_2$ by the formula: 
\begin{equation}\label{formula} 
f (z)=\sum_{\alpha \geq 0} c_\alpha \, z^\alpha
\end{equation}
where the series is absolutely convergent for each $z = (z_j)_{j\geq 1}\in \Omega_2$, as the pointwise product of two square-summable sequences. Indeed, using 
an Euler type formula, we get for $z\in \Omega_2$:
\begin{displaymath} 
\sum_{\alpha \geq 0}|z^{\alpha}|^2 = \prod_{j=1}^\infty (1 -|z_j|^2)^{-1}<\infty \, ,
\end{displaymath} 
and hence, by the Cauchy-Schwarz inequality:
\begin{displaymath} 
\sum_{\alpha \geq 0}|c_\alpha \, z^\alpha| \leq 
\bigg( \sum_{\alpha \geq 0}|c_\alpha|^2\bigg)^{1/2} \bigg(\sum_{\alpha \geq 0}|z^\alpha|^2\bigg)^{1/2} <  \infty \, .
\end{displaymath} 
If $\alpha\in E$ and $z\in \Omega_2$, we have set, as usual, $z^\alpha=\prod_{j\geq 1} z_{j}^{\alpha_j}$. 
\smallskip

This shows that $L^2_E (\T^\infty)$ can be identified with $H^{2}(\Omega_2)$, the Hardy-Hilbert space of analytic functions 
$f(z)=\sum_{\alpha \geq 0} c_{\alpha} \, z^{\alpha}$ on $\Omega_2$ with 
\begin{displaymath} 
\Vert f \Vert^{2} :=\sum_{\alpha \geq 0}|c_\alpha|^2 < \infty \, . 
\end{displaymath} 
This setting  is customary in connection with Dirichlet series (see \cite{COGA}). 

%The space $H$ thus appears both as $L^{2}_{E}=:H^{2}(\T^\infty)$ and as a Hilbert space of analytic functions on $\Omega$.
\medskip

In this paper, for a technical reason appearing  below in the proof of Proposition~\ref{simple}, we will consider, instead of $\Omega_2 = \D^\infty \cap \ell_2$, 
the sub-domain:
\begin{displaymath} 
\Omega= \D^\infty \cap \ell_1 \, , 
\end{displaymath} 
i.e. the \textit{open} subset of $\ell^1$ formed by the sequences: 
\leavevmode\vspace{-0.5 \baselineskip}
\begin{displaymath} 
\qquad z = (z_n)_{n\geq 1} \quad \text{such that} \quad |z_n|< 1\, ,  \forall \, n \geq 1, \quad \text{and}\quad \sum_{n=1}^\infty |z_n|<\infty \, ,
\end{displaymath} 
\vskip - 5 pt \noindent
and the restrictions to $\Omega$ of the functions $f \in H^{2}(\Omega_2)$. We denote $H^{2}(\Omega)$ the space of such restrictions. 

Hence $f \in H^{2}(\Omega)$ if and only if:
\begin{displaymath} 
f (z)=\sum_{\alpha \geq 0} c_{\alpha} \, z^{\alpha} \quad \text{with } z \in \Omega \, ,
\end{displaymath} 
and $\Vert f \Vert^{2} :=\sum_{\alpha \geq 0}|c_\alpha|^2 < \infty$. 
\smallskip

We now identify the space $L^2_E (\T^\infty)$  with the space $H^{2}(\Omega)$.
\medskip
 
We more generally define Hardy spaces $H^p (\Omega)$, for $1\leq p<\infty$, in the usual way:
\begin{displaymath} 
H^{p} = H^{p}(\Omega)= \{f \colon \Omega \to \C \, ; \  \Vert f \Vert_p<\infty\} \, , 
\end{displaymath} 
where $f$ is analytic in $\Omega$ and $\Vert f \Vert_p = \sup_{0<r<1}M_{p}(r,f)=\lim_{r\to 1^-} M_{p}(r,f)$ with:
\begin{displaymath} 
M_{p}(r,f)=\bigg(\int_{\T^\infty}|f (r z)|^p \, dm(z)\bigg)^{1/p}, \quad 0<r<1 \, . 
\end{displaymath} 
We have $\Vert f \Vert=\Vert f\Vert_2$. Moreover, $H^q$ contractively embeds into $H^p$ for $p<q$.
%\smallskip
  
%%%%%%%%%%%%%%%
\subsection{Singular numbers}

 We begin with a reminder of operator-theoretic facts. We recall that the approximation numbers  $a_{n}(T) = a_n$ of an operator $T \colon H\to H$ (with 
$H$ a Hilbert space) are defined by:
\begin{displaymath} 
a_n = \inf_{\textrm{rank}\, R< n} \Vert T - R\Vert \, .
\end{displaymath} 

According to a 1957's result of Allahverdiev (see \cite{CARL-STEPH}, page~155), we have $a_n=s_n$, the $n$-th singular number of $T$. 
We also recall a basic result due to H.~Weyl and one obvious consequence:
\begin{theoreme}\label{hermann} 
Let $T \colon H\to H$ be a compact operator with eigenvalues $(\lambda_n)$ rearranged in decreasing order and singular numbers $(a_n)$. Then:
\begin{displaymath} 
\qquad \prod_{j=1}^n |\lambda_j|\leq \prod_{j=1}^n a_j \quad \text{for all } n\geq 1 \,.
\end{displaymath} 
As a consequence:
$$|\lambda_{2n}|^2\leq a_1 a_n.$$
\end{theoreme}

%%%%%%%%%%%%%%%%%%%%%%%%%%%%%%%%%%%%%%%%%%%%%%%%%%%%%%%%%%%%%
\subsection{Spectra of projective tensor products}

The following operator-theoretic result will play a basic role in the sequel.
Let $E_1,\ldots, E_n$ be Banach spaces and let $E=\otimes_{i=1}^n E_i$ their \textit{projective} tensor product (the only tensor product we shall use). 
If $T_i\in \mathcal{L}(E_i)$, we define as usual their projective tensor product $T=\otimes_{i=1}^n T_i\in \mathcal{L}(E)$ by its action on the atoms of $E$, 
namely:
\begin{displaymath} 
T(\otimes_{i=1}^n x_i)=\otimes_{i=1}^n T_{i}(x_i) \, .
\end{displaymath} 
Denote in general $\sigma(x)$ the spectrum of $x\in \mathcal{A}$ where $\mathcal{A}$ is a unital Banach algebra. We recall 
(\cite{RU}, chap.11, Theorem 11.23) the following result.
\begin{lemme}\label{rud} 
Let $\mathcal{A}$ be a unital Banach algebra, and $x_1,\ldots, x_n$ be \emph{pairwise commuting} elements of $\mathcal{A}$. Then:
\begin{displaymath} 
\sigma(x_1\cdots x_n)\subseteq \prod_{i=1}^n \sigma(x_i) \, .
\end{displaymath} 
\end{lemme}
Here, $\prod_{i=1}^n \sigma(x_i)$ is the product in the Minkowski sense, namely:
\begin{displaymath} 
\prod_{i=1}^n \sigma(x_i)=\bigg\{\prod_{i=1}^n \lambda_i : \lambda_i\in \sigma(x_i)\bigg\} \, .
\end{displaymath} 

As a consequence, we then have the following lemma due to Schechter, which we prove under a weakened form, sufficient here, and which is indeed already in 
\cite{ARGALI} (we just add a few details because this is a central point in our estimates).
\begin{lemme} 
Let $F$ be a Banach space, $T_1,\ldots,T_n\in \mathcal{L}(F)$ and $T=\otimes_{i=1}^n T_i$. Then $\sigma(T)\subset \prod_{i=1}^n \sigma(T_i)$.
\end{lemme}
\begin{proof} 
To save notation, we assume $n=2$. Let $x_1=T_1\otimes I_2$ and $x_2=I_1\otimes T_2$ where $I_i$ is the identity of $E_i$. Clearly, 
\begin{displaymath} 
x_1 x_2 = x_2 x_1=T_1\otimes T_2=T \quad \text{and}\quad \sigma(x_i)=\sigma (T_i)
\end{displaymath} 
where the spectrum of $x_i$ is in the Banach algebra $\mathcal{L}(E)$ and that of $T_i$ in $\mathcal{L}(E_i)$. Lemma~\ref{rud} now gives:
\begin{displaymath} 
\sigma(T) =\sigma(x_1 x_2)\subseteq \sigma(x_1) \, \sigma(x_2)= \sigma(T_1)\, \sigma(T_2) \, , 
\end{displaymath} 
hence the result.
\end{proof}
%
%%%%%%%%%%%%%%%%%%%%%%%%%%%%%%%%%%%%%%%%%%%%%%%%%%%%%%%%%%%%%%%%%%%%%%%%%%%%%%%
\subsection{ Schur maps and  composition operators}

We now pass to some general facts on composition operators $C_\varphi$, defined by $C_{\varphi}(f) = f \circ \varphi$, associated with a Schur map, namely  
a \emph{non-constant} analytic self-map $\varphi$ of $\Omega$. We say that $\phi$ is a \emph{symbol} for $H^2 (\Omega)$ if $C_\phi$ is a bounded 
linear operator from $H^2 (\Omega)$ into itself. 

The differential $\varphi '(a)$ of $\varphi$ at some point $a\in \Omega$ is a bounded linear map $\varphi'(a) \colon \ell^1\to \ell^1$. 
\begin{definition}
The symbol $\varphi$ is said to be \emph{truly infinite-dimensional} if  the differential $\varphi'(a)$ is an \emph{injective} linear map from $\ell^1$ into itself 
for at least one point $a\in \Omega$. 
\end{definition}

In finite dimension, this amounts to saying that $\varphi(\Omega)$ has non-void interior. 
\vskip 5pt

We have the following general result.
\begin{proposition}\label{simple} 
Let $(\varphi_j)_{j \geq 1}$ be a sequence of analytic self-maps of \ $\D$ such that $\sum_{j\geq 1}|\varphi_{j}(0)|<\infty$. Then, the mapping  
$\varphi \colon \Omega\to \C^\infty$ defined by the formula $\varphi(z)=(\varphi_{j}(z_j))_{j\geq 1}$ maps $\Omega$ to itself  and is a symbol for 
$H^2 (\Omega)$.
\end{proposition}

\begin{proof} 
First, the Schwarz inequality:
\begin{displaymath} 
|\varphi_{j}(z_j)-\varphi_{j}(0)|\leq2 \, |z_j | 
\end{displaymath} 
shows that $\varphi(z)\in \Omega$ when $z\in \Omega$. To see that $\varphi$ is moreover a symbol for $H^2 (\Omega)$, we use the fact (\cite{COMA}) that:
\begin{equation}\label{usefac} 
\Vert C_{\varphi_j}\Vert\leq \sqrt{\frac{1+|\varphi_{j} (0)|}{1-|\varphi_{j} (0)|}}\cdot 
\end{equation}
Now, by the separation of variables and Fubini's theorem, we easily get:
\begin{equation} \label{usefac-bis}
\Vert C_{\varphi}\Vert\leq \prod_{j=1}^\infty \Vert C_{\varphi_j}\Vert<\infty \, .
\end{equation} 
As $\sum_{j\geq 1}|\varphi_{j}(0)| < \infty$, by hypothesis, the infinite product 
\begin{displaymath} 
\prod_{j\geq 1}\sqrt{\frac{1+|\varphi_{j} (0)|}{1-|\varphi_{j} (0)|}} 
\end{displaymath} 
converges and, in view of \eqref{usefac} and \eqref{usefac-bis}, $C_\phi$ is bounded.
\end{proof}
\smallskip

We also have the following useful fact.
\begin{lemme}\label{trans} 
The automorphisms of  $\Omega$ act transitively on $\Omega$ and define bounded composition operators on $H^{2}(\Omega)$.
\end{lemme}
\begin{proof}  Let  $a=(a_j)_j \in \Omega$ and let $\Psi_a \colon \Omega\to \C^\infty$ be defined by:
\begin{displaymath} 
\Psi_{a}(z)=\big(\Phi_{a_j}(z_j)\big)_{j\geq 1}
\end{displaymath} 
where in general $\Phi_u \colon \D\to \D$ is defined by $\Phi_{u}(z)= (z - u) / (1-\overline{u}z)$. The Schwarz lemma gives 
$|\Phi_{a_j}(z_j) + a_j|\leq 2|z_j|$, and shows that $\Psi_a$ maps $\Omega$ to itself. Clearly, $\Psi_a$ is an automorphism of $\Omega$ with inverse 
$\Psi_{-a}$ and $\Psi_{a}(a)=0$. The fact that the composition operator $C_{\Psi_a}$ is bounded on $H^{2}(\Omega)$ is a consequence of 
Proposition~\ref{simple}. 
\end{proof}
%
%%%%%%%%%%%%%%%%%%
%\section{A general lower bound}
%%%%%%%%%%%%%%%%%%%%%

%We will here consider analytic self-maps $\varphi=(\varphi_j)_{j\geq 1}$ of $\Omega$ (i.e. for all $z\in \Omega$, one has $|\varphi_{j}(z)|<1$  and $\sum_{j\geq 1}|\varphi_{j}(z)|^2<\infty$) and the formally associated composition operator. Some of them are bounded as shown by the following simple criterion:
%%%%%%%%%

%%%%%%%%%%%%%%%%%%%%%%%%%%%%%%%%%%%%%%%%%%%%%%%%%%%%%%%%%%%%%%%%%%%%%

\section{Spectrum of compact composition operators}

We begin with the following definition, following \cite{LEF}.
\begin{definition}
Let $\phi \colon \Omega \to \Omega$ be a truly infinite-dimensional symbol. We say that $\phi$ is \emph{compact} if $\overline{\varphi(\Omega)}$ is a 
compact subset of $\Omega$. 
\end{definition}

We then have the following result.
\goodbreak
\begin{lemme}\label{compact} 
If $\varphi \colon \Omega\to \Omega$ is a \emph{compact mapping}, then: \par\smallskip
$1)$ $C_\varphi \colon H^{2}(\Omega)\to H^{2}(\Omega)$ is bounded and moreover compact. \par\smallskip

$2)$ If $a\in \Omega$ a fixed point of $\varphi$,  $\varphi'(a)\in \mathcal{L}(\ell^1)$ is a compact operator.
\end{lemme}
\begin{proof} $1)$ follows from a H.~Schwarz type criterion via an Ascoli-Montel type theorem: every sequence $(f_n)$ of $H^{2}(\Omega)$ bounded 
in norm contains a subsequence which converges uniformly on compact subsets of $\Omega$. Indeed, we have the following (\cite{CHAE}, chap.~17, p.~274): 
if $A$ is a locally bounded set of holomorphic functions on $\Omega$, then $A$ is locally equi-Lipschitz, namely every point $a\in \Omega$ has a neighourhood 
$U\subset \Omega$ such that:
\begin{displaymath} 
z,w\in U \quad \text{and}\quad f \in A \quad\Longrightarrow \quad|f(z) - f(w)|\leq C_{A,U} \, \Vert z - w\Vert \, .
\end{displaymath} 
The Ascoli-Montel theorem easily follows from this. Then, if $f_n\in H^{2}(\Omega)$ converges weakly to $0$, it converges uniformly to $0$ on compact subsets
of $\Omega$; in particular on $\overline{\varphi(\Omega)}$. This means that $\Vert C_{\varphi}(f_n)\Vert_\infty =\Vert f_n \circ\varphi\Vert_\infty \to 0$, 
implying $\Vert f_n\circ \varphi\Vert_2 \to 0$ and the compactness of $C_\varphi$. \par

Actually, $C_\varphi$ is compact on every Hardy space $H^{p}(\Omega),\ 1\leq p\leq\infty$. This observation will be useful later on.\par
\smallskip

For $2)$, we may indeed dispense ourselves with the invariance of $a$, and force $a=0$ to be a fixed point of $\varphi$. Indeed, we can replace $\varphi$ by 
$\psi = \Psi_b\circ \varphi\circ \Psi_a$ where $b=\varphi (a)$ is arbitrary, and use Lemma~\ref{trans} as well as the ideal property of compact linear operators. 
We set $A=\varphi'(0)$.  Expanding each coordinate $\varphi_j$ of $\varphi$ in a series of homogeneous polynomials, we may write (since $\varphi(0)=0$):
\begin{displaymath} 
\varphi(z)=\sum_{|\alpha|=1} c_\alpha z^\alpha + \sum_{s=2}^\infty \bigg(\sum_{|\alpha|=s} c_\alpha z^\alpha\bigg) 
= A(z)+\sum_{s=2}^\infty \bigg(\sum_{|\alpha|=s} c_\alpha z^\alpha\bigg) \, ,
\end{displaymath} 
where $c_\alpha=(c_{\alpha,j})_{j\geq 1}\in \C^\infty$. We clearly have (looking at the Fourier series of $\varphi (z \, \e^{i\theta})$): 
\begin{equation}\label{bohr} 
\Vert z\Vert_1<1 \quad \Longrightarrow \quad z\in \Omega \quad \Longrightarrow \quad 
A(z) = \frac{1}{2\pi}\int_{0}^{2\pi}\varphi(z \, \e^{i\theta}) \, \e^{-i\theta} \, d\theta \, .
\end{equation}
Since $\varphi$ is compact, this clearly implies, with $B$ the open unit ball of $\ell^1$, that $A(B)$ is totally bounded, proving the compactness of $A$.   
\end{proof} 

The following extension of results of \cite{MCCL}, then \cite{ARGALI} and \cite{CLA}, which themselves extend a theorem of 
G.~K\"onigs (\cite{Shap-livre}, p.~93) will play an essential role for lower bounds of approximation numbers. 
\begin{theoreme}\label{plusbluff} 
Let $\varphi \colon \Omega\to\Omega$ be a compact symbol. Assume there is $a\in \Omega$ such that $\varphi(a)=a$ and that  
$\varphi'(a)\in \mathcal{L}(\ell^1)$ is \emph{injective}. Then, the spectrum of $C_\varphi \colon H^{2}(\Omega)\to H^{2}(\Omega)$ is exactly formed by 
the numbers $\lambda^\alpha$, $\alpha \in \N^{(\infty)}$, and $0, 1$, where $(\lambda_j)_{j \geq 1}$ denote the eigenvalues of $A:=\varphi'(a)$ and:
\begin{displaymath} 
\qquad \qquad \quad \lambda^\alpha =\prod_{j \geq 1}\lambda_{j}^{\alpha_j} \qquad \text{if} \quad \alpha=(\alpha_j)_{j \geq 1} \in \N^{(\infty)} \, .
\end{displaymath} 
\end{theoreme}
\begin{proof} 
This is proved in \cite{ARGALI} for the unit ball $B_E$ of an arbitrary Banach space $E$ and for the space $H^{\infty}(B_E)$, in four steps which are the 
following:  \par\smallskip
$1.$ If $\varphi(B_E)$ lies strictly inside $B_E$ (namely if $\varphi(B_E)\subseteq rB_E$ for some $r<1$), in particular when $\varphi$ is compact, 
$\varphi$  has a unique fixed point $a\in B_E$, according to a theorem of Earle and Hamilton. \par\smallskip

$2.$ The spectrum of $C_\varphi$ contains the numbers $\lambda$ where $\lambda$ is an eigenvalue of $\varphi'(a)$ or $\lambda=0,1$. \par\smallskip

$3.$ It is then proved that the spectrum of $C_\varphi$ contains  the numbers $\lambda^\alpha$ and $0, 1$.  \par\smallskip

$4.$ It is finally proved that  spectrum of $C_\varphi$ is contained in  the numbers $\lambda^\alpha$ and $0, 1$. \par\medskip

Here, handling with the domain $\Omega$, we see that:  \par\smallskip

$1.$ True or not for $\Omega$, the Earle-Hamilton theorem is not needed since we will force, by a change of the compact symbol $\varphi$ in another 
compact symbol $\psi=\Psi_b\circ \varphi\circ \Psi_a$, the point $0$ to be a fixed point. Moreover $A=\psi'(0)$ is injective if $\varphi'(a)$ is, since $\Psi'_a$ 
and $\Psi'_b$ are invertible.  \par\smallskip

$2.$ The second step (non-surjectivity) is valid for any domain and for $H^{2}(\Omega)$, or $H^{p}(\Omega)$, in exactly the same way.  \par\smallskip

$3.$ The third step consists of proving $\{\lambda^{\alpha}\}\subseteq \sigma(C_\varphi)$. \par\smallskip

For that purpose, assume that $\lambda^\alpha=\prod_{l=1}^m \lambda_l \neq 0$ with $\lambda_l$ an eigenvalue of $\varphi'(0)$ and  with repetitions 
allowed. As we already mentioned, under the assumption of compactness of $\varphi$, $C_\varphi$ is compact on $H^{p}(\Omega)$ as well, for any $p\geq 1$. 
We take here $p=2m$. Step $2$ provides us with non-zero functions $f_{i}\in H^{p}(\Omega)$ such that $f_i\circ \varphi=\lambda_i f_i$, $1\leq i\leq m$, 
since for the compact operator $C_\varphi \colon H^p\to H^p$, non-surjectivity implies non-injectivity. Let $f=\prod_{1\leq i\leq m} f_i$. Then, using the 
integral representation of the norm and the H\"older inequality, we see that $f\in H^{2}(\Omega)$, $f\neq 0$ and $f\circ \varphi=\lambda^{\alpha}f$, 
proving our claim. \par\smallskip

$4.$ The fourth step is valid as well, with a slight simplification: we have to show that, if $\mu\neq 1$ is not of the form $\lambda^\alpha$, then 
$C_\varphi - \mu I$ is injective. Let $f\in H^{2}(\Omega)$ satisfying $f\circ \varphi=\mu f$ and let:
\begin{displaymath} 
f(z)=\sum_{m=0}^\infty \frac{d^{m} f (0)}{m!}(z^m) 
\end{displaymath} 
be the Taylor expansion of $f$ about $z=0$ (observe that $\Omega$ is a Reinhardt domain). As usual,  $d^{m}f(0) =: L_m$ is an $m$-linear symmetric form on 
$F=\ell^1$ and the notation $L_{m}(z^m)$ means $L_{m}(z, z, \ldots, z)$. 

Observe that $L_m$ can be isometrically identified with an element (denoted $\overline{L_m}$) of $\mathcal{L}(F^{\otimes n})$ defined by the formula:
\begin{displaymath} 
\overline{L_m}(x_1\otimes\cdots \otimes x_n) = L_{m}(x_1,\ldots, x_m) \, .
\end{displaymath} 
We will prove by induction that $L_{n}=0$ for each $n$. For this, we can avoid the appeal to transposes of \cite{ARGALI} as follows: if the result holds for 
$L_{m}$ with $m<n$, one gets (comparing the terms in $z^n$ in both members of $f\circ \varphi=\mu f$):
\begin{equation}\label{onegets} 
\mu A=A\circ B \quad \text{where} \quad A=\overline{L_n} \quad \text{and}\quad B=\varphi'(0)^{\otimes n} \, . 
\end{equation}
That is $A (B -\mu I)=0$ where $I$ is the identity map of $F^{\otimes n}$. Now, $B-\mu I$ in invertible in $\mathcal{L}(F)$ by 
Lemma~\ref{plusbluff}, so that $A=A(B-\mu I)(B-\mu I)^{-1}=0$. \par\smallskip

The proof is complete. 
\end{proof}
\smallskip

The following theorem summarizes and exploits the preceding theorem. Possibly, some restrictions can be removed, and we could only assume the compactness of 
$C_\varphi$, not of $\varphi$ itself. After all, in dimension one, there are symbols $\varphi$ with $\Vert \varphi\Vert_\infty=1$ for which 
$C_\varphi \colon H^2\to H^2$ is compact.

\begin{theoreme}\label{geninf} 
Let $\varphi \colon \Omega\to \Omega$ be a truly infinite-dimensional \textnormal{compact}  mapping of $\Omega$. Then: \par\smallskip
$1)$ $C_\varphi:H^{2}(\Omega)\to H^{2}(\Omega)$ is bounded and even compact.  \par\smallskip

$2)$ $A=\varphi'(0)$ is  compact.  \par\smallskip

$3)$ No $\delta>0$ can exist such that $a_{n}(C_\varphi)\leq C\, \e^{-c \, n^\delta}$\ for all $n\geq 1$. More precisely, the numbers $a_n$ satisfy:
\begin{equation}\label{mopr}
\sum_{n\geq 1} \frac{1}{\log^{p}(1/a_n)}=\infty \quad \text{for all}\quad p<\infty \, . 
\end{equation} 
\end{theoreme}
\begin{proof} 
The proof is based on the previous Theorem~\ref{plusbluff}. Without loss of generality, we can assume that $\varphi(0)=0$ and 
$\varphi'(0)$ is injective, by using a point $a$ at which $\varphi'(a)$ is injective, and then the fact that automorphisms of $\Omega$ act transitively on $\Omega$,  
act boundedly on $H^{2}(\Omega)$, and the ideal property of approximation numbers. More precisely, we pass to $\Psi=\Psi_b\circ \varphi\circ \Psi_a$ with 
$b=\varphi(a)$ and get: 
\begin{displaymath} 
\Psi(0)=0 \quad \text{and}\quad  \Psi'(b)=\Psi'_{b}(b) \, \varphi'(a) \, \Psi'_{a}(0)   
\end{displaymath} 
injective, since $\Psi'_{b}(b)$ and $\Psi'_{a}(0)$ are, and $\Psi_a$ and $\Psi_b$  are automorphisms of $\Omega$.
\smallskip

We now have the following simple but crucial lemma.
\begin{lemme}\label{cruci}
Whatever the choice of the numbers $\lambda_j$ with $0<|\lambda_j|<1$, denoting by $(\delta_n)_{n\geq 1}$ the non-increasing rearrangement of the numbers 
$\lambda^\alpha$, one has: 
\begin{displaymath} 
\sum_{n\geq 1} \frac{1}{\log^{p}(1/\delta_n)}=\infty \quad \text{for all} \quad p<\infty \, . 
\end{displaymath} 
\end{lemme}

\begin{proof} [Proof of the Lemma]
For any positive integer $p$, we set:
\begin{displaymath} 
q= 2p \, ,\quad   \log 1/|\lambda_j|=A_j \, ,
\end{displaymath} 
and we use that:
\begin{displaymath} 
\sum_{1\leq j\leq q} \alpha_j\,A_j\leq \bigg(\sum_{1\leq j\leq q}\alpha_{j}^2\bigg)\bigg(\sum_{1\leq j\leq q} A_{j}^2\bigg) 
=:C_q \,  \bigg(\sum_{1\leq j\leq q}\alpha_{j}^2\bigg) = C_q\Vert \alpha\Vert^2 \, ,
\end{displaymath} 
where $\Vert\, . \,\Vert$ stands for the euclidean norm in $\R^q$. We then get: 
\begin{align*} 
\sum_{n\geq 1} \frac{1}{\log^{p}(1/\delta_n)} 
& =\sum_{\alpha>0} \frac{1}{\log^{p}(1/|\lambda^\alpha|)} \\
& \geq \sum_{\alpha_j\geq 1, \, 1\leq j\leq q}\frac{1}{\log^{p}(1/|\lambda_1^{\alpha_1}|\cdots 1/|\lambda_q^{\alpha_q}|)} \\
& =\sum_{\alpha_j\geq 1,\, 1\leq j\leq q}\frac{1}{(\alpha_1 A_1+\cdots+ \alpha_q A_q)^p} \\
& \geq C_{q}^{-p} \sum_{\alpha_j\geq 1,\, 1\leq j\leq q}\frac{1}{(\alpha_{1}^2 +\cdots +\alpha_{q}^2 )^p} \\
& = C_{q}^{-p} \sum_{\alpha_j\geq 1,\, 1\leq j\leq q}\frac{1}{\Vert\alpha\Vert^{q}}=\infty \,,
\end{align*} 
because:
\begin{displaymath} 
\int_{x\in \R^q, \, \Vert x\Vert \geq 1}\frac{1}{\Vert x\Vert^{q}} \, dx = c_q \int_{1}^\infty \frac{r^{q-1}}{r^{q}} \, dr = \infty \, .
\end{displaymath} 
This proves the lemma.
\end{proof}

This can be transferred to the approximation numbers $a_n=a_{n}(C_\varphi)$ to end the proof of Theorem~\ref{geninf}.
Indeed, we know from Lemma~\ref{cruci} that the non-increasing rearrangement $(\delta_n)$ of the eigenvalues $\lambda^\alpha$ of $C_\varphi$ satisfies
\begin{displaymath} 
\sum_{n\geq 1} \frac{1}{\log^{p}(1/\delta_{n})}=\infty \, . 
\end{displaymath} 
Since a divergent series of non-negative and non-increasing numbers $u_n$ satisfies $\sum u_{2n}=\infty$, we further see  that:
\begin{displaymath} 
\sum_{n\geq 1} \frac{1}{\log^{p}(1/\delta_{2n})}=\infty \quad  \text{for all}\quad p<\infty \, . 
\end{displaymath} 
Moreover, by  Theorem~\ref{hermann} we have:
\begin{equation}\label{hw}
\bigg(\frac{1}{2\log 1/\delta_{2 n}}\bigg)^p\leq \bigg(\frac{1}{\log 1/ (a_1 a_n)}\bigg)^p\cdot
\end{equation}
Since $1/ (\log 1/ a_1 a_n) \sim 1/ (\log 1/a_n)$, Lemma~\ref{cruci} then  gives the result. This clearly prevents an inequality of the form 
$a_n\leq C \, \e^{- c \, n^\delta}$ for some positive numbers $c, C, \delta$ and all $n\geq 1$. Indeed, this would imply:
\begin{displaymath} 
\sum_{n\geq 1} \frac{1}{\log^{p}(1/a_n)}<\infty \quad  \text{for} \quad p>1/\delta \, , 
\end{displaymath} 
contradicting \eqref{mopr}. 
\end{proof}
\smallskip

\noindent {\bf Remarks.} Let us briefly comment on the assumptions in Theorem~\ref{geninf}. \par\smallskip

$1)$ We do not need the Earle-Hamilton theorem under our assumptions. The Schauder-Tychonoff  theorem gives the existence (if not the uniqueness) of a 
fixed point for $\varphi$ in $\Omega$ (bounded and convex). \par\smallskip

$2)$ The Earle-Hamilton theorem is in some sense more general (for analytic maps) since it remains valid when $\overline{\varphi(\Omega)}$ is only 
assumed to lie strictly inside $\Omega$, i.e. when $\varphi(\Omega)\subseteq r\Omega$ for some $r<1$. But this assumption does not ensure the compactness of 
$C_{\varphi}$ as indicated by the simple example $\varphi(z)=rz$, $0<r<1$. The coordinate functions $z\mapsto z_n$ converge weakly to $0$, while 
$\Vert C_{\varphi}(z_n)\Vert_{H^{2}(\Omega)}=r$. \par\smallskip

$3)$ The mere assumption that $\overline{\varphi(\Omega)}$ is compact is not sufficient either. Juste take:
\begin{displaymath} 
\varphi(z)= \bigg(\frac{1+z_1}{2},0,\ldots, 0,\ldots \bigg) \, .
\end{displaymath} 
Since the composition operator $C_{\varphi_1}$ associated with $\varphi_{1}(z)=\frac{1+z}{2}$ is notoriously non-compact on $H^{2}(\D)$, neither is 
$C_\varphi$ on $H^{2}(\Omega)$. Yet, $\overline{\varphi(\Omega)}$ is obviously compact in $\ell^1$.
\goodbreak

%%%%%%%%%%%%%%%%%%%%%%%%%%%%%%%%%%%%%%%%%%%%%%%%%%%%%%%%%%%%%%%%%%%%%%%%%%%
\section{Possible upper bounds}
%%%%%%%%%%%%%%%%

Recall that  $\Omega=\D^\infty \cap\ell^1$. 
%We will  say in that section that $\varphi=(\varphi_i)_{i\geq 1}$ is \emph{upper triangular} if the 
%$i$th-component only depends on the remaining variables $z_i,\ z_{i+1},\ldots$, that is:
%
%\begin{align*}
%\varphi_{1}(z) & = \varphi_{1}(z_1, z_2,\ldots) \\
%\varphi_{2}(z) &= \varphi_{2}(z_2, z_3,\ldots) \\
%\cdots \cdots & \cdots \cdots \cdots \cdots \cdots \\
%\varphi_{i}(z) &= \varphi_{i}(z_i, z_{i+1},\ldots) \\
%\cdots \cdots & \cdots \cdots \cdots \cdots \cdots 
%\end{align*}
%
%
%When $\varphi$ is upper triangular, the jacobian matrix 
%$(D_{j}\varphi_{i})_{i, j}= \big(\frac{\partial \varphi_i}{\partial z_j} \big)_{i, j}$ is also upper triangular, namely 
%$\frac{\partial \varphi_i}{\partial z_j}=0$ for $j<i$. In the sequel, we assume that $\varphi(0)=0$ and we set:
% 
%\begin{equation}\label{set} 
%a_{i,j}= \Big(\frac{\partial \varphi_i}{\partial z_j} \Big)(0) \, ,\qquad \lambda_j=a_{j,j} \, .
%\end{equation} 
%
%
%We  begin with a transference result. It is stated and proved for diagonal maps only, but probably valid for upper triangular maps as well.  
%
%%%%%%%%%%%%%%%%%%%%%%%%%%%
\subsection{A general example}

\begin{theoreme}\label{soft} 
Let $\varphi((z_j)_j)=(\lambda_j z_j)_j$ with $ |\lambda_j|<1$ for all $j$, so that $\varphi(\Omega)\subseteq \Omega$ and $\varphi'(0)$ is the diagonal 
operator with eigenvalues $\lambda_j$, $j \geq 1$, on the canonical basis of $\ell^1$. Let $p>0$. Then:
\begin{displaymath} 
(\lambda_j)_j \in \ell^p \quad \Longrightarrow \quad C_\varphi\in S_p \, . 
\end{displaymath} 

In particular, there exist truly infinite-dimensional symbols on $\Omega$ such that the composition operator 
$C_\varphi \colon H^{2}(\Omega)\to H^{2}(\Omega)$ is in all Schatten classes $S_p$, $p>0$. %In particular, $a_{n}(C_\varphi)$ is fast decreasing:
%$$a_{n}(C_\varphi)\leq C_A n^{-A} \hbox{\ for all}\ A>0.$$
\end{theoreme}
\begin{proof} 
Since $C_\varphi$ is diagonal on the orthonormal basis $(z^\alpha)_\alpha$ of the Hilbert space $H^{2}(\Omega)$, with 
$C_{\varphi}(z^\alpha)=\varphi^\alpha$, its approximation numbers are the non-increasing rearrangement of the moduli of eigenvalues $\lambda^\alpha$, 
so that an Euler product-type  computation gives:
\begin{displaymath} 
\sum_{n=1}^\infty a_{n}^p=\sum_{\alpha\in E}|\lambda^{\alpha}|^p 
=\sum_{\alpha_j\in \N}\prod_{j\geq 1}|\lambda_{j}|^{p\alpha_j}=\prod_{j=1}^\infty(1-|\lambda_{j}|^p)^{-1}<\infty \, .  
\end{displaymath} 
To obtain $C_\varphi\in \bigcap_{p>0} S_p$, just take $\lambda_n= \e^{- n}$. This ends the proof.
\end{proof}
%\medskip

%%%%%%%%%%%%%%%%%%%%%%%%%%%%%%
\subsection{A sharper upper bound}

By making a more quantitative study, we can prove the following result.
\begin{theoreme}\label{supscha} 
For any $0<\delta<1$, there exists a compact composition operator on $H^{2}(\Omega)$, with a truly infinite-dimensional symbol, such that, for some positive 
constants $c,C, b$, we have:
\begin{displaymath} 
a_{n}(C_\varphi)\leq C \, \exp \big( -c\, \e^{b \, (\log n)^{\delta}} \big) \, .
\end{displaymath} 
\end{theoreme}
\begin{proof} 
Take the same operator $C_\varphi$ as in Theorem~\ref{soft}, with $\lambda_n = \e^{- A_n}$ where the positive numbers $A_n$ have to be adjusted. 
Its approximation numbers $a_N$ are then the non-increasing rearrangement of the sequence of numbers $(\eps_n)_n := (\lambda^{\alpha})_\alpha$. 
This suggests using a generating function argument, namely considering $\sum \eps_n x^n$, but the rearrangement perturbs the picture. Accordingly, we follow 
a sligthly different route. Fix an integer $N\geq 1$ and a real number $r > 0$. Observe that, following the proof of Theorem \ref{soft}: 
\begin{displaymath} 
N\, a_{N}^{r}\leq \sum_{n=1}^N a_{n}^{r}\leq \sum_{n=1}^\infty a_{n}^{r}=\prod_{n=1}^\infty(1-e^{-rA_n})^{-1}.
\end{displaymath} 

First, consider the simple example $A_n=n$. We get:
\begin{displaymath} 
N\, a_{N}^{r}\leq \eta \, (\e^{-r}) 
\end{displaymath} 
where $\eta$ is the Dedekind eta function (see \cite{CHA}) given by:
\begin{displaymath} 
\qquad \qquad \eta(x)=\prod_{n=1}^\infty (1- x^n)^{-1}=\sum_{n=0}^\infty p(n) \, x^n \, ,\qquad |x|< 1 \, , 
\end{displaymath} 
where $p(n)$ is the number of partitions of the integer $n$. It is well-known (\cite{CHA}, Ch.~7, p.~169) that 
$\eta \, (\e^{-r})\leq \e^{D /r}$ with $D =\pi^2 /6$, so that:
\begin{displaymath} 
a_N \leq \exp\bigg(\frac{D}{r^2} -\frac{\log N}{r}\bigg) \, .
\end{displaymath} 
Optimizing with $r = 2D / \log N$, we get:
\begin{displaymath} 
a_N \leq \exp (- c\log^{2} N) \, ,
\end{displaymath} 
with $c = 1 /4D$. This is more precise than Theorem~\ref{soft}.
\smallskip

We now show that if $A_n$ increases faster, we can achieve the decay of Theorem~\ref{supscha}. As before, we get in general:
\begin{equation}\label{general} 
a_N \leq \inf_{x >1} \big( \exp \, [x(\log F(x^{-1}) - \log N)] \big)\, ,
\end{equation}
where 
\begin{displaymath} 
F(r) = \prod_{n=1}^\infty(1 - \e^{-r A_n})^{-1} \, .
\end{displaymath} 
We have:
\begin{displaymath} 
\log F(r)=\sum_{n=1}^\infty \bigg( \sum_{m=1}^\infty \frac{\e^{- r m A_n}}{m}\bigg) 
= \sum_{m = 1}^\infty \frac{1}{m} \bigg( \sum_{n = 1}^\infty \e^{- r m A_n} \bigg) \, .
\end{displaymath} 
Now, take $A_n = \e^{n^\alpha}$ where $\alpha>0$ is to be chosen. We have:
\begin{displaymath} 
\sum_{n=1}^\infty \e^{-r m \, \e^{n^\alpha}} \leq \int_{0}^\infty \e^{- r m \, \e^{t^\alpha}}dt =:I_{m}(r) \, .
\end{displaymath} 
Standard estimates now give, for $r <1$:
\begin{align*}
I_{m}(r) 
& =\int_{1}^\infty \e^{- r m x} \frac{1}{\alpha} (\log x)^{ \frac{1}{\alpha}-1}\,\frac{dx}{x}
=\int_{rm}^\infty \e^{-y} \frac{1}{\alpha} \bigg(\log \frac{y}{rm} \bigg)^{ \frac{1}{\alpha} - 1}\,\frac{dy}{y} \\
& \lesssim \bigg(\log \frac{1}{r} \bigg)^{ \frac{1}{\alpha} - 1}\int_{rm}^\infty \e^{- y} \frac{dy}{y}
\lesssim \e^{- r m} \bigg(\log \frac{1}{r} \bigg)^{ \frac{1}{\alpha}} \, ,
\end{align*}
so that: 
\begin{displaymath} 
\log F(r) \lesssim (\log 1/r)^{ \frac{1}{\alpha}}\sum_{m=1}^\infty m^{-1} \e^{-r m} \lesssim (\log 1/r)^{ \frac{1}{\alpha}+1} \, .
\end{displaymath} 
Going back to \eqref{general}, we get, for some constant $C>0$, and for $x=1/r>1$:
\begin{displaymath} 
a_N \leq C  \exp\big[C \, x \big( (\log x)^{ \frac{1}{\alpha}+1}-\log N \big)\big] \, .
\end{displaymath} 
Adjusting $x=x_N>1$ so as to have $(\log x)^{ \frac{1}{\alpha}+1}=\log N-1$, that is:
\begin{displaymath} 
x_N = \exp\big[(\log (N/\e))^{\frac{\alpha}{\alpha+1}}\big] \, , 
\end{displaymath} 
we get $a_N \leq C \, \e^{- c \, x_N}$, which is the claimed result with $\delta= \alpha / (\alpha+1)$. \par 

This $\delta$ can be taken arbitrarily in $(0,1)$ by choosing $\alpha$ suitable, and we are done. 
\end{proof}

\noindent{\bf Remark.} Of course, $\delta=1$ is forbidden, because this would give $a_n \leq C\, \e^{- c \, n^b}$, implying:
\begin{displaymath} 
\sum_{n =1}^\infty  \frac{1}{(\log 1/a_n)^p} \lesssim \sum_{n=1}^\infty n^{- b \, p} <\infty \, ,
\end{displaymath} 
for large $p$, and contradicting Theorem~\ref{geninf}.

\bigskip

Daniel Li \\ 
Univ. Artois, Laboratoire de Math\'ematiques de Lens (LML) EA~2462, \& F\'ed\'eration CNRS Nord-Pas-de-Calais FR~2956 \\
Facult\'e Jean Perrin, Rue Jean Souvraz, S.P.\kern 1mm 18 \\
F-62\kern 1mm 300 LENS, FRANCE \\
daniel.li@euler.univ-artois.fr
\smallskip

Herv\'e Queff\'elec \\
Univ. Lille Nord de France, USTL \\ 
Laboratoire Paul Painlev\'e U.M.R. CNRS 8524 \& F\'ed\'eration CNRS Nord-Pas-de-Calais FR~2956 \\
F-59\kern 1mm 655 VILLENEUVE D'ASCQ Cedex, FRANCE \\
Herve.Queffelec@univ-lille1.fr
\smallskip
 
Luis Rodr{\'\i}guez-Piazza \\
Universidad de Sevilla, Facultad de Matem\'aticas, Departamento de An\'alisis Matem\'atico \& IMUS \\ 
Apartado de Correos 1160 \\ 
41\kern 1mm 080 SEVILLA, SPAIN \\
piazza@us.es

\end{document}